\documentclass{amsart}

\usepackage[bookmarksnumbered,colorlinks, plainpages]{hyperref}
\hypersetup{colorlinks=true,linkcolor=red, anchorcolor=green, citecolor=cyan, urlcolor=red, filecolor=magenta, pdftoolbar=true}

\usepackage{amsmath, amsthm, amscd, amsfonts, amssymb, graphicx, color, enumerate, xcolor, mathrsfs, latexsym}

\theoremstyle{plain}
\newtheorem{theorem}{Theorem}[section]
\newtheorem{lemma}[theorem]{Lemma}
\newtheorem{corollary}[theorem]{Corollary}
\newtheorem{conjecture}[theorem]{Conjecture}
\newtheorem{question}[theorem]{Question}

\theoremstyle{definition}

\theoremstyle{remark}

\newcommand{\normal}{\trianglelefteq}

\newcommand{\ol}{\overline}
\renewcommand{\leq}{\leqslant} 
\renewcommand{\geq}{\geqslant}

\newcommand{\Aut}{\mathrm{Aut}}

\newcommand{\gen}[1]{\langle#1\rangle}

\newcommand{\set}[1]{\left\{#1\right\}}

\newcommand{\D}{\mathcal{D}}

\begin{document}
\title{Finite groups with five relative commutativity degrees}
\author{M. Farrokhi D. G.}
\date{}
\address{Department of Mathematics, Institute for Advanced Studies in Basic Sciences (IASBS), and the Center for Research in Basic Sciences and Contemporary Technologies, IASBS, Zanjan 66731-45137, Iran}
\email{m.farrokhi.d.g@gmail.com, farrokhi@iasbs.ac.ir}
\begin{abstract}
We classify all finite groups with five relative commutativity degrees. Also, we give a partial answer to our previous conjecture on a lower bound of the number of relative commutativity degrees of finite groups.
\end{abstract}
\maketitle
\section{Introduction}
A finite group is either abelian or non-abelian, but not all non-abelian groups share the same commutativity relation among their elements. Roughly speaking, nilpotent groups seems to be more commutative than solvable groups, and solvable groups seems even more commutative than non-solvable groups. To compare groups via their commutativity of elements, one can count all commuting pairs in a finite group $G$ and normalize it by dividing this number by the number of all pairs. This quantity defined explicitly as
\[d(G):=\frac{\#\{(x,y)\in G\times G\ :\ xy=yx\}}{|G|^2}\]
is known as the \textit{commutativity degree} of $G$. Erd\"{o}s and Turan \cite{pe-pt} defined the commutativity degree of groups in their study of symmetric groups and showed that it satisfies the identity
\[d(G)=\frac{k(G)}{|G|},\]
where $k(G)$ denotes the number of conjugacy classes of $G$. This formula is used to give various lower and upper bounds for $d(G)$ in the literature. The most simple upper bound for $d(G)$ is given by Gustafson \cite{whg} in 1973 who showed that $d(G)\leq 5/8$ for all non-abelian groups with equality if and only if $G/Z(G)$ is the Klein four group. The next remarkable and significant result is due to Rusin \cite{djr} in 1979 who classified all finite groups with commutativity degrees greater than $11/32$. Since then the commutativity degree of finite groups is studied actively and we may refer the interested reader to \cite{fm-dm-ans, se, ive-bs, rmg-grr, eh-dm-ans, ph, pl, pl-hnn-yy} for some major contributions to the field.

While the commutativity degree can be applied to distinguish between groups but it is not so strong to reveal internal structure of groups in general. For instance, 
\[d(A_4)=d(D_{18})=\frac{1}{3}\]
while $A_4$ and $D_{18}$ have quite different structures. One way to overcome this problem is to work on local structure of groups by looking at subgroups of a group and how their elements commute with other elements of the group under consideration. Accordingly, one can define the \textit{relative commutativity degree} of a subgroup $H$ of a finite group $G$ as
\[d(H,G):=\frac{\#\{(h,g)\in H\times G\ :\ hg=gh\}}{|H||G|}.\]
The above quantity is introduced by Erfanian, Rezaei, and Lescot \cite{ae-rr-pl} in 2007 where the authors apply it to show, among other results, the following monotonic property of (relative) commutativity degrees
\[d(G)\leq d(K,G)\leq d(H,G)\leq d(H,K)\leq d(H),\]
when $G$ is a finite group and $H$ and $K$ are subgroups of $G$ with $H\leq K$.

Barzegar, Erfanian, and Farrokhi \cite{rb-ae-mfdg} consider the set $\D(G)$ of all relative commutativity degrees of a finite group $G$, namely
\[\D(G)=\{d(H,G)\ :\ H\leq G\},\]
and study the groups $G$ when $\D(G)$ is small. It is evident that $\D(G)$ is a singleton if and only if $G$ is abelian. They show that there is no finite groups $G$ with $\D(G)$ possessing only two elements, and obtain the following classification of finite groups $G$ for which $\D(G)$ has three elements.
\begin{theorem}[{\cite[Theorems 3.1 and 3.2]{rb-ae-mfdg}}]\label{|D(G)|=3}
Let $G$ be a finite group. Then $|\D(G)|=3$ if and only if $G$ satisfies one of the following cases:
\begin{itemize}
\item[(i)]$G/Z(G)\cong C_p\times C_p$ for some prime $p$ (nilpotent case). Then
\[\D(G)=\set{1,\ \frac{2p-1}{p^2},\ \frac{p^2+p-1}{p^3}}.\]
\item[(ii)]$G/Z(G)\cong C_p\rtimes C_q$ is a non-abelian group of order $pq$ for some distinct primes $p$ and $q$ (non-nilpotent case). Then
\[\D(G)=\set{1,\ \frac{p+q-1}{pq},\ \frac{p+q^2-1}{pq^2}}.\]
\end{itemize}
\end{theorem}

The above results are extended by Erfanian and Farrokhi \cite{ae-mfdg} by classifying all finite groups $G$ with $\D(G)$ containing four elements.
\begin{theorem}[{\cite[Theorems 3.1 and 3.5]{ae-mfdg}}]\label{|D(G)|=4}
Let $G$ be a finite group. Then $|\D(G)|=4$ if and only if $G$ satisfies one of the following cases:
\begin{itemize}
\item[(i)]$G/Z(G)$ is a $p$-group of order $p^3$ and $G$ has no abelian maximal subgroups (nilpotent case). Then
\[\D(G)=\set{1,\ \frac{p^2+p-1}{p^3},\ \frac{2p^2-1}{p^4},\ \frac{p^2+p^3-1}{p^5}}.\]
\item[(ii)]$G/Z(G)\cong(C_p\times C_p)\rtimes C_q$ is a minimal Frobenius group and the Sylow $p$-subgroup of $G$ is abelian, where $p$ and $q$ are distinct primes (non-nilpotent case). Then
\[\D(G)=\set{1,\ \frac{p+q-1}{pq},\ \frac{p^2+q-1}{p^2q},\ \frac{p^2+q^2-1}{p^2q^2}}.\]
\end{itemize}
\end{theorem}

Here by a minimal Frobenius group we mean a Frobenius group none of its proper subgroups are Frobenius. Notice that a Frobenius group $G$ is minimal if and only if its kernel is elementary abelian, its complements are cyclic groups of prime orders, and both the kernel and complements are maximal subgroups of $G$.

For a finite group $G$, we can write the set $\D(G)$ as
\[\D(G)=\{d_0,\ d_1,\ \ldots,\ d_n\}\]
where $1=d_0>d_1>\cdots>d_n=d(G)$. Recently, the above results on finite groups $G$ satisfying $|\D(G)|\leq4$ are generalized by Farrokhi and Safa \cite{mfdg-hs} by describing those subgroups $H$ of a finite group $G$ satisfying $d(H,G)\geq d_3$. They show that $H/H\cap Z(G)$ is a product of at most $i$ primes when $d(H,G)=d_i\geq d_3$ and pose the following conjecture:
\begin{conjecture}\label{|D(G)|>=Omega(G/Z(G))+1?}
Let $G$ be a finite group and $H$ be a subgroup of $G$ with $d(H,G)=d_k$. Then $|H/Z(H,G)|$ is a product of at most $k$ primes. As a result, $|G/Z(G)|$ is a product of at most $|\D(G)|-1$ primes.
\end{conjecture}

Also, they state a weaker version of the above conjecture as
\begin{conjecture}\label{|D(G)|>=l_M(G/Z(G))+1?}
Every finite group $G$ satisfies
\[|\D(G)|\geq l_M\left(\frac{G}{Z(G)}\right)+1,\]
where $l_M(X)$ denotes the maximum length of chains of subgroups of the group $X$.
\end{conjecture}

In this paper, we shall classify all finite groups with five relative commutativity degrees. We divide these groups into two families. For nilpotent groups, we have
\begin{theorem}\label{nilpotent}
Let $G$ be a finite nilpotent group. Then $|\D(G)|=5$ if and only if one of the following holds:
\begin{itemize}
\item[(i)]$|G/Z(G)|=p^3$ and $G$ has an abelian maximal subgroup. Then
\[\D(G)=\set{1,\frac{2p-1}{p^2},\frac{p^2+p-1}{p^3},\frac{3p-2}{p^3},\frac{2p^2-1}{p^4}}.\]
\item[(ii)]$|G/Z(G)|=p^4$ and $G$ has two conjugacy class sizes $1$ and $p^m$ for some fixed $m\in\{1,2,3\}$. Then
\[\D(G)=\set{1,\frac{p^m+p-1}{p^{m+1}},\frac{p^m+p^2-1}{p^{m+2}},\frac{p^m+p^3-1}{p^{m+3}},\frac{p^m+p^4-1}{p^{m+4}}}.\]
\end{itemize}
\end{theorem}

The second family contains non-nilpotent groups and described as follows:
\begin{theorem}\label{non-nilpotent}
Let $G$ be a finite non-nilpotent group. Then $|\D(G)|=5$ if and only if one of the following holds:
\begin{itemize}
\item[(i)]$G/Z(G)\cong C_p\rtimes C_{q^2}$ is a Frobenius group. Then
\[\D(G)=\set{1,\ \frac{p+q-1}{pq},\ \frac{p+q^2-1}{pq^2},\ \frac{p+q^3-1}{pq^3},\ \frac{p+q^4-1}{pq^4}}.\]
\item[(ii)]$G/Z(G)\cong C_{p^2}\rtimes C_q$ is a Frobenius group. Then 
\[\D(G)=\set{1,\ \frac{p+q-1}{pq},\ \frac{p^2+q-1}{p^2q},\ \frac{p^2+q^2+pq-p-q}{p^2q^2},\ \frac{p^2+q^2-1}{p^2q^2}}.\]
\item[(iii)]$G/Z(G)\cong(C_p\times C_p)\rtimes C_q$ is a Frobenius group with a normal subgroup of order $p$ and the Sylow $p$-subgroup of $G$ is abelian. Then 
\[\D(G)=\set{1,\ \frac{p+q-1}{pq},\ \frac{p^2+q-1}{p^2q},\ \frac{p^2+q^2+pq-p-q}{p^2q^2},\ \frac{p^2+q^2-1}{p^2q^2}}.\]
\item[(iv)]$G/Z(G)\cong (C_p\times C_p)\rtimes C_q$ is a minimal Frobenius group and the Sylow $p$-subgroup of $G$ is non-abelian. Then, either $G/Z(G)\cong A_4$ for which
\[\D(G)=\set{1,\ \frac{7}{12},\ \frac{1}{2},\ \frac{3}{8},\ \frac{7}{24}},\]
or $p>q$ and 
\[\D(G)=\set{1,\ \frac{p^2+q-1}{p^2q},\ \frac{pq+p-1}{p^2q},\ \frac{pq+p^2-1}{p^3q},\ \frac{p^2q+p^2-1}{p^4q}}.\]
\item[(v)]$G/Z(G)\cong (C_p\times C_p\times C_p)\rtimes C_q$ is a minimal Frobenius group and the Sylow $p$-subgroup of $G$ is abelian. Then
\[\D(G)=\set{1,\ \frac{p+q-1}{pq},\ \frac{p^2+q-1}{p^2q},\ \frac{p^3+q-1}{p^3q},\ \frac{p^3+q^2-1}{p^3q^2}}.\]
\end{itemize}
Here, $p$ and $q$ denote distinct primes.
\end{theorem}

One observe that all groups $G$ with at most five relative commutativity degrees satisfy Conjecture \ref{|D(G)|>=Omega(G/Z(G))+1?}. In \cite{mfdg-hs} the authors show that all supersolvable groups admit Conjecture \ref{|D(G)|>=Omega(G/Z(G))+1?} too. Hence all these groups satisfy Conjecture \ref{|D(G)|>=l_M(G/Z(G))+1?}. As a partial result, in Lemma \ref{|D(G)|>=Omega(G/Z(G))+1}, we also show that all finite groups whose nontrivial elements of their central factor groups have prime power orders satisfy Conjecture \ref{|D(G)|>=l_M(G/Z(G))+1?}.
\section{Preliminary results}
In this section, we recall/prove a set of useful tools which we shall use frequently in our proofs. We note that $d(H,G)=d(HZ,G)$ for every subgroup $H$ and every central subgroup $Z$ of a finite group $G$. We shall use this fact without further citing.
\begin{lemma}[{\cite[Lemma 2.1]{rb-ae-mfdg}}]\label{d(K,G)<=d(H,G)}
Let $G$ be a finite group and $H,K$ be subgroups of $G$ such that $H\leq K$. Then $d(K,G)\leq d(H,G)$ with equality if and only if $K=HC_K(g)$ for all $g\in G$
\end{lemma}
\begin{corollary}\label{d(<x>,G)<d(<x^p>,G)}
Let $G$ be a finite non-abelian group and $x\in G$. If the order of $xZ(G)$ in $G/Z(G)$ is divisible by a prime $p$, then $d(\gen{x},G)<d(\gen{x^p},G)$.
\end{corollary}
\begin{lemma}[{\cite[Lemma 2.1]{mfdg-hs}}]\label{HnonabelianKabelian}
Let $G$ be a finite non-abelian group. If $K\leq G$ is non-abelian and $H\leq K$ is abelian, then $d(K,G)<d(H,G)$.
\end{lemma}
\begin{lemma}[{\cite[Lemma 2.2]{mfdg-hs}}]\label{CHgABH}
Let $G$ be a finite non-abelian group, $H$ be a subgroup of $G$ and $g\in G\setminus C_G(H)$. If $A,B$ are subgroups of $H$ such that $C_H(g)\leq A<B\leq H$, then $d(B,G)<d(A,G)$.
\end{lemma}
\begin{lemma}\label{xHCGx}
Let $G$ be a finite group and $x\in G$. If $C_G(x)$ is non-abelian and $H$ is an abelian subgroup of $G$ such that $\gen{x}\subset H\subset C_G(x)$ and $H\not\subseteq Z(C_G(x))$, then 
\[d(C_G(x),G)<d(H,G)<d(\gen{x},G).\]
\end{lemma}
\begin{proof}
By Lemma \ref{HnonabelianKabelian}, $d(C_G(x),G)<d(H,G)$. If $d(H,G)=d(\gen{x},G)$, then $H=\gen{x}C_H(g)$ for all $g\in G$. If $g\in C_G(x)$, then $x\in C_H(g)$ and $H=C_H(g)$, from which it follows that $H\subseteq Z(C_G(x))$, a contradiction. Thus $d(H,G)<d(\gen{x},G)$.
\end{proof}

We shall use the following result to reduce the nilpotent groups under consideration to the class of $p$-groups.
\begin{lemma}\label{D(HxK)}
Let $H$ and $K$ be two finite groups with coprime orders. Then
\begin{itemize}
\item[(1)]$\D(H\times K)=\D(H)\D(K)$ is the set of products of elements of $\D(H)$ by elements of $\D(K)$;
\item[(2)]$\D(H)\cap\D(K)=\{1\}$.
\end{itemize}
\end{lemma}

While the above lemma gives the simple lower bound 
\[|\D(H\times K)|\geq|\D(H)|+|\D(K)|\]
for $|\D(H\times K)|$ when $H$ and $K$ are finite groups with coprime orders, we believe that a more stronger result should holds for any two such groups.
\begin{conjecture}
For any two finite groups $H$ and $K$ of coprime orders, we have
\[|\D(H\times K)|=|\D(H)|\times|\D(K)|.\]
\end{conjecture}

Notice that the above conjecture is not valid in general. The smallest counter-example is $(A_4,S_4)$ for which $|\D(A_4\times S_4)|=|\D(A_4)||\D(S_4)|-3$. Two rather paradoxical examples are
\[|\D(S_4\times S_4)|=|\D(S_4)|^2-17\ \text{and}\ |\D(S_5\times S_5)|=|\D(S_5)|^2+24\]
showing that not only the difference between $|\D(H\times K)|$ and $|\D(H)||\D(K)|$ can be large but also they cannot be compared in general. So, we may ask the following question:
\begin{question}
How the difference
\[|\D(H\times K)|-|\D(H)||\D(K)|\]
grows with respect to $|\D(H)|$ and $|\D(K)|$?
\end{question}

All over this paper, $G$ denotes the group we are working on and $\ol{G}$ stand for the factor group $G/Z(G)$. Accordingly, for a subgroup $H$ of $G$ and element $g\in G$, $\ol{H}$ and $\ol{g}$ stand for $HZ(G)/Z(G)$ and $gZ(G)$.
\section{Nilpotent groups}
Our classification of finite nilpotent groups with four relative commutativity degrees relies on two classes of groups; $p$-groups with an abelian maximal subgroup and $p$-groups whose all non-central elements have the same conjugacy class sizes. In both cases, we can simply compute all relative commutativity degrees and count them.
\begin{lemma}\label{abelian maximal subgroup}
Let $G$ be a non-abelian finite $p$-group with an abelian maximal subgroup. If $|G/Z(G)|=p^n$, then $|\D(G)|=2n-1$ and
\[\D(G)=\set{\frac{p^i+p-1}{p^{i+1}},\ \frac{p^{j-1}+p-1}{p^{j+1}}+\frac{p-1}{p^n}\ :\ 0\leq i<n,\ 1<j\leq n}.\]
\end{lemma}
\begin{proof}
Let $M$ be the unique abelian maximal subgroup of $G$. Clearly, $C_G(g)=M$ for all $g\in M\setminus Z(G)$, and $C_G(g)=\gen{Z(G),g}$ has order $p|Z(G)|$ for all $g\in G\setminus M$. Let $H$ be a subgroup of $G$ containing $Z(G)$ and $|\ol{H}|=p^i$. If $H\subseteq M$, then $0\leq i<n$ and
\[d(H,G)=\frac{|Z(G)||G|+(|H|-|Z(G)|)|M|}{|H||G|}=\frac{p^i+p-1}{p^{i+1}}.\]
Also, if $H\not\subseteq M$, then $1\leq i\leq n$ and $|\ol{H\cap M}|=p^{i-1}$. Thus
\begin{align*}
d(H,G)&=\frac{|Z(G)||G|+(|H\cap M|-|Z(G)|)|M|+(|H|-|H\cap M|)p|Z(G)|}{|H||G|}\\
&=\frac{p^{i-1}+p-1}{p^{i+1}}+\frac{p-1}{p^n}.
\end{align*}
On the other hand, if 
\[\frac{p^i+p-1}{p^{i+1}}=\frac{p^{j-1}+p-1}{p^{j+1}}+\frac{p-1}{p^n},\]
for some $0\leq i<n$ and $1\leq j\leq n$, then a simple verification shows that $i=n-1$ and $j=1$. Therefore $|\D(G)|=2n-1$, as required.
\end{proof}
\begin{lemma}\label{two conjugacy class sizes}
Let $G$ be a finite $p$-group whose non-central elements have conjugacy classes of the same size $p^m$. If $|G/Z(G)|=p^n$, then $|\D(G)|=n+1$ and
\[\D(G)=\set{\frac{p^m+p^i-1}{p^{m+i}}\ :\ 0\leq i\leq n}.\]
\end{lemma}
\begin{proof}
Let $H$ be a subgroup of $G$ containing $Z(G)$. If $\ol{H}=p^i$, then
\[d(H,G)=\frac{|Z(G)||G|+(|H|-|Z(G)|)|G|/p^m}{|H||G|}=\frac{p^m+p^i-1}{p^{m+i}},\]
as required.
\end{proof}

In what follows, $Z$ and $Z_g$ stand for the subgroups $Z(G)$ and $\gen{g}\cap Z(G)$ of a group $G$ for all elements $g\in G$. Utilizing the above two lemmas, we can classify all nilpotent groups with five relative commutativity degrees. 
\newline\newline
\textbf{\textit{Proof of Theorem \ref{nilpotent}.} }
Suppose $|\D(G)|=5$. Let $G=P_1\times\cdots\times P_k$ be the factorization of $G$ into the direct product of Sylow $p$-subgroups. If $P_i$ and $P_j$ are non-abelian for some $i\neq j$, then the paragraph after Lemma \ref{D(HxK)} shows that $|\D(G)|\geq6$, which is a contradiction. Therefore $G/Z(G)$ is a $p$-group. By Theorem \ref{|D(G)|=3}(i) and \cite[Theorem 2.3]{mfdg-hs}, we must have $|G/Z(G)|=p^3$ or $p^4$. 

If $|G/Z(G)|=p^3$, then $G/Z(G)$ has an abelian maximal subgroup by Theorem \ref{|D(G)|=4}(i) so that Lemma \ref{abelian maximal subgroup} yields $G$ is a group of type (i) with $\D(G)$ as given in the theorem. 

Now, suppose that $|G/Z(G)|=p^4$. If $G$ has two conjugacy class sizes, then we apply Lemma \ref{two conjugacy class sizes} to show that $G$ is a group of type (ii) with $\D(G)$ as in the theorem. For the rest of the proof, we further assume that $G$ has at least three conjugacy class sizes. By Lemma \ref{abelian maximal subgroup}, $G$ has no abelian maximal subgroups. We have two cases to consider:

Case 1. $\exp(\ol{G})=p^2$. Let $\ol{x}\in\ol{G}$ be an element of order $p^2$. Clearly, $C_G(x)=\gen{Z(G),x}$ as $G$ has no abelian maximal subgroups. We show that $C_G(x^p)$ is a maximal subgroup of $G$. Suppose on the contrary that $C_G(x^p)=C_G(x)$. Let $M$ be a maximal subgroup of $G$ containing $x$. Since $d(C_G(x),G)=d(\gen{x},G)$, by Corollary  \ref{d(<x>,G)<d(<x^p>,G)} and Lemma \ref{CHgABH},
\[d(G)<d(M,G)<d(\gen{x},G)<d(\gen{x^p},G)<1.\]
If $g\in G$ is such that $|\ol{C_G(g)}|=p$, and $C_G(g)\subset M_1\subset M_2\subset G$ for some subgroups $M_1$ and $M_2$ of $G$, then Lemma \ref{CHgABH} along with the fact that $d(C_G(g),G)=d(\gen{g},G)$ yield
\[d(G)<d(M_2,G)<d(M_1,G)<d(\gen{g},G)<1.\]
Thus $d(\gen{g},G)=d(\gen{x^p},G)$, which implies that $|\ol{C_G(g)}|=|\ol{C_G(x^p)}|\geq p^2$, a contradiction. Hence $|\ol{C_G(g)}|\geq p^2$ for all $g\in G$. As a result, $G$ contains an element $y$ such that $C_G(y)$ is a maximal subgroup of $G$. Notice that $C_G(y)$ is non-abelian and hence we should have $|\ol{y}|=p$.  By Lemma \ref{HnonabelianKabelian},
\[d(G)<d(C_G(y),G)<d(\gen{y},G)<1.\]
Since $d(\gen{y},G)\neq d(\gen{x^p},G)$, it follows that $d(\gen{y},G)=d(\gen{x},G)$. On the other hand,
\[d(\gen{x},G)=\frac{|Z_x||G|+(|x|-|Z_x|)p^2|Z|}{|x||G|}=\frac{2p^2-1}{p^4}\]
and
\[d(\gen{y},G)=\frac{|Z_y||G|+(|y|-|Z_y|)p^3|Z|}{|y||G|}=\frac{2p-1}{p^3},\]
which imply that $d(\gen{y},G)\neq d(\gen{x},G)$, a contradiction. Therefore $C_G(x^p)$ is a maximal subgroup of $G$, and by Corollary \ref{d(<x>,G)<d(<x^p>,G)} and Lemma \ref{HnonabelianKabelian},
\[d(G)<d(C_G(x^p),G)<d(\gen{x},G)<d(\gen{x^p},G)<1.\]
Next we show that $C_G(g)$ is a maximal subgroup of $G$ for all $g\in G$ satisfying $|\ol{g}|=p$. To this end, let $g\in G$ be such that $|\ol{g}|=p$, $|\ol{C_G(g)}|=p^c\neq p^3$, and $M$ be a maximal subgroup of $G$ containing $C_G(g)$. Then Lemma \ref{CHgABH} yields
\[d(G)<d(M,G)<d(C_G(g),G)\leq d(\gen{g},G)<1.\]
Since $d(\gen{g},G)\neq d(\gen{x^p},G)$, it follows that $d(\gen{g},G)=d(\gen{x},G)$. On the other hand,
\[d(\gen{x},G)=\frac{|Z_x||G|+(|x^p|-|Z_x|)p^3|Z|+(|x|-|x^p|)p^2|Z|}{|x||G|}=\frac{3p-2}{p^3}\]
and
\[d(\gen{g},G)=\frac{|Z_g||G|+(|g|-|Z_g|)p^c|Z|}{|g||G|}=\frac{p^3+p-1}{p^4}\ \text{or}\ \frac{p^2+p-1}{p^3},\]
which imply that $d(\gen{g},G)\neq d(\gen{x},G)$, a contradiction. Therefore $|\ol{C_G(g)}|=p^3$ for all $g\in G$ such that $|\ol{g}|=p$. If $\ol{G}$ has a subgroup $\ol{H}\cong C_p\times C_p$, then
\[d(H,G)=\frac{|Z||G|+(|H|-|Z|)p^3|Z|}{|H||G|}=\frac{p^2+p-1}{p^3}.\]
On the other hand, if $g\in G\setminus C_G(H)$, then $C_H(g)$ is a non-central subgroup of $G$ and hence
\[d(H,G)<d(C_H(g),G)<1\]
by Lemma \ref{CHgABH}. Let $a,b$ be the number of cycles of order $p^2$ and $p$ in $\ol{M}$, respectively with $M$ a maximal subgroup of $G$ containing $H$. Then 
\[d(M,G)=\frac{|Z||G|+p(p-1)ap^2|Z|^2+(p-1)bp^3|Z|^2}{|M||G|}=\frac{p+(p-1)(a+b)}{p^4},\]
which implies that $d(H,G)\neq d(M,G)$. Note that if $d(H,G)=d(M,G)$, then we get $a+b=p(p+2)$. As $1+p(p-1)a+(p-1)b=p^3$ (the order of $|\ol{M}|$), we obtain $a=-1$, which is a contradiction. Similarly, we can show that $d(H,G)\neq d(G)$. Thus $d(H,G)=d(\gen{x},G)$, which is impossible. Therefore $\ol{G}$ has a unique subgroup of order $p$ showing that $\ol{G}\cong Q_{16}$ (see \cite[Theorem 5.3.6]{djsr}). Accordingly, $\ol{G}$ contains an element of order $8$ so that $G$ has an abelain maximal subgroup, which is a contradiction. 

Case 2. $\exp(\ol{G})=p$. Let $g\in G\setminus Z(G)$. We consider three possibilities for the order of $C_G(g)$.

(a) $|\ol{C_G(g)}|=p$. Let $A_g$ and $B_g$ be subgroups of $G$ such that $C_G(g)\subset A_g\subset B_g\subset G$. Since $d(C_G(g),G)=d(\gen{g},G)$, Lemma \ref{CHgABH} yields
\begin{equation}\label{centralizer1}
d(G)<d(B_g,G)<d(A_g,G)<d(\gen{g},G)<1.
\end{equation}

(b) $|\ol{C_G(g)}|=p^2$. Let $C_g$ be a maximal subgroup of $G$ containing $C_G(g)$. If $d(C_G(g),G)=d(\gen{g},G)$, then $C_G(g)=\gen{g}C_{C_G(g)}(g')$ for all $g'\in G$. Hence $g'\in C_G(g'')$ for some $g''\in C_G(g)\setminus\gen{Z(G),g}$ for all $g'\in G\setminus C_G(g)$. Then
\[\ol{G}=\bigcup_{\ol{1}\neq\gen{\ol{g''}}\leq\ol{C_G(g)}}\ol{C_G(g'')},\]
from which it follows that 
\[p^4\leq p^2+p(p^3-p^2)=p^4-p^3+p^2<p^4,\]
a contradiction. Therefore, Lemmas \ref{d(K,G)<=d(H,G)} and \ref{CHgABH} give us
\begin{equation}\label{centralizer2}
d(G)<d(C_g,G)<d(C_G(g),G)<d(\gen{g},G)<1.
\end{equation}

(c) $|\ol{C_G(g)}|=p^3$. Let $g'\in C_G(g)\setminus\gen{Z(G),g}$ and $g''\in C_G(g)\setminus\gen{Z(G),g,g'}$. If $D_g:=\gen{Z(G),g,g'}$, then $C_{C_G(g)}(g')=D_g$ and $C_{D_g}(g'')=\gen{Z(G),g}$ so that 
\[d(C_G(g),G)\neq d(D_g,G)\neq d(\gen{g},G)\]
by Lemma \ref{d(K,G)<=d(H,G)}. Thus
\begin{equation}\label{centralizer3}
d(G)<d(C_G(g),G)<d(D_g,G)<d(\gen{g},G)<1.
\end{equation}

Now, from \eqref{centralizer1}, \eqref{centralizer2}, and \eqref{centralizer3}, it follows that $d(\gen{x},G)=d(\gen{y},G)$ and hence $|C_G(x)|=|C_G(y)|$ for all non-central elements $x,y$ of $G$. This shows that $G$ has only two conjugacy class sizes, which contradicts our assumption. The proof is complete.
\hfill$\Box$
\section{Non-nilpotent groups}
This section is devoted to non-nilpotent finite groups with five relative commutativity degrees. To classify these groups, first we show that non-trivial elements in central factor group of any such group have prime power orders.
\begin{lemma}\label{G/Z(G) has prime power order elements}
Let $G$ be a finite non-nilpotent group. If $|\D(G)|=5$, then the order of every element of $G/Z(G)$ is equal to $p$ or $p^2$ for some prime $p$.
\end{lemma}
\begin{proof}
First assume that $\ol{G}$ has an element $\ol{x}$ of order $pqr$ for some primes $p,q,r$. By Corollary \ref{d(<x>,G)<d(<x^p>,G)},
\begin{equation}\label{|x|=pqr}
d(\gen{x},G)<d(\gen{x^p},G)<d(\gen{x^{pq}},G)<1.
\end{equation}
If $M$ is a maximal subgroup of $G$ containing $C_G(x)$, then
\[d(G)<d(M,G)\leq d(C_G(x),G)\leq d(\gen{x},G),\]
from which, in conjunction with \eqref{|x|=pqr}, it follows that $d(M,G)=d(C_G(x),G)=d(\gen{x},G)$. Hence, $C_G(x)=M$ is an abelian maximal subgroup of $G$ by Lemmas \ref{HnonabelianKabelian} and \ref{CHgABH}. Now, the properties of $C_G(x)$ yield
\[d(\gen{x},G)=\frac{1}{pqr}+\left(1-\frac{1}{pqr}\right)\frac{1}{[G:C_G(x)]}\]
and
\[d(C_G(x),G)=\frac{1}{[C_G(x):Z(G)]}+\left(1-\frac{1}{[C_G(x):Z(G)]}\right)\frac{1}{[G:C_G(x)]},\]
so that $[C_G(x):Z(G)]=pqr$ and consequently $C_G(x)=\gen{Z(G),x}$. We have two cases to consider:

Case 1. $C_G(x)\normal G$. Then $|G/Z(G)|=pqrs$ for some prime $s$. Let $y\in G\setminus C_G(x)$ be an $s$-element and $H=\gen{Z(G),x^p,y}$. Since $H$ is non-abelian, Lemma \ref{HnonabelianKabelian} shows that $d(H,G)<d(\gen{y},G)$. On the other hand,
\[d(\gen{y},G)=\frac{1}{s}+\left(1-\frac{1}{s}\right)\frac{1}{pqr}=d(\gen{x},G),\]
which implies that $d(H,G)<d(\gen{x},G)$. Thus $d(H,G)=d(G)$ so that $G=HC_G(y)=H$ by Lemma \ref{d(K,G)<=d(H,G)}, a contradiction. 

Case 2. $C_G(x)\not\normal G$. Then $N_G(C_G(x))=C_G(x)$ and $C_G(x)\cap C_G(x)^g=Z(G)$ for all $g\in G$. It follows that $\ol{G}$ is a Frobenius group with complement $\ol{C_G(x)}$. Let $\ol{K}$ be the kernel of $\ol{G}$. Then $\ol{K}$ is an elementary abelian $s$-group for some prime $s$. If $y\in K\setminus Z(G)$, then 
\[\gen{y}\leq C_G(y)\leq K<K\gen{x^{pq}}<K\gen{x^p}<G.\] 
Hence, Lemma \ref{CHgABH} yields
\[d(G)<d(K\gen{x^p},G)<d(K\gen{x^{pq}},G)<d(K,G)\leq d(C_G(y),K)\leq d(\gen{y},G)<1,\]
which implies that $d(K,G)=d(\gen{y},G)$. Thus $K=\gen{y}C_K(x)=\gen{Z(G),y}$ by Lemma \ref{d(K,G)<=d(H,G)}, which implies that $|\ol{K}|=s$ is a prime. Now, proceeding the same arguments as in Case 1 leads us to a contradiction. Therefore $\ol{G}$ has no elements of order $pqr$.

For the rest of the proof, we suppose that $\ol{G}$ has an element $\ol{x}$ of order $pq$, where $p$ and $q$ are distinct primes. Then $\ol{x}=\ol{a}\ol{b}$, where $|\ol{a}|=p$, $|\ol{b}|=q$, and $ab=ba$. If $d(\gen{a},G)\neq d(\gen{b},G)$, then since $d(\gen{x},G)<d(\gen{a},G),d(\gen{b},G)$, by replacing $x^{pq}$ by $x^q$ and noting that $d(K\gen{x^p},G)\neq d(K\gen{x^q},G)$ via direct computations, we reach to a contradiction by the same arguments as above. Thus, we suppose in addition that $d(\gen{a},G)=d(\gen{b},G)$. It follows that $|C_G(a)|\neq|C_G(b)|$ and hence $C_G(x)$ is not a maximal subgroup of $G$.

If $C_G(x)$ is non-abelian and $H$ is an abelian non-central subgroup of $C_G(x)$ containing $\gen{x}$ properly, then we obtain
\[d(C_G(x),G)<d(H,G)<d(\gen{x},G)\]
by Lemma \ref{xHCGx}, which implies that $|\D(G)|>5$, a contradiction. Thus $C_G(x)$ is abelian so that $\ol{C_G(x)}\cong C_p^m\times C_q^n$ for some $m,n\geq1$ as $\ol{G}$ has no elements of orders a product of three primes . Let $M$ be a maximal subgroup of $G$ containing $C_G(x)$, and $\ol{H}$ be a Sylow subgroup of $\ol{C_G(x)}$. It is evident that all non-central elements of $H$ have the same centralizer sizes as $C_G(x)$ is abelian. Then, by Lemmas \ref{d(K,G)<=d(H,G)} and \ref{CHgABH}, we get
\[d(G)<d(M,G)<d(C_G(x),G)<d(H,G)\leq d(\gen{g},G)<1\]
for every $g\in H\setminus Z(G)$. Hence $d(H,G)=d(\gen{g},G)$, from which it follows that $H=\gen{Z(G),g}$. Therefore $C_G(x)=\gen{Z(G),x}$. 

Now, we show that $\ol{C_G(a)}$ is a $\{p,q\}$-group. Suppose on the contrary that $\pi(\ol{C_G(a)})\neq\{p,q\}$ and $\ol{c}\in\ol{C_G(a)}$ is an element of prime order $r\neq p,q$. Then $\ol{ac}$ is an element of order $pr$ in $\ol{G}$ and, as above, we should have $d(\gen{c},G)=d(\gen{a},G)$. Also, we must have $d(\gen{ac},G)=d(\gen{x},G)$, from which it follows that $|\ol{C_G(a)}|=pqr$. By Schur-Zassenhaus theorem (see \cite[Theorem 9.1.2]{djsr}), we get $\ol{C_G(a)}=\gen{\ol{a}}\times\gen{\ol{b},\ol{c}}$ so that $\gen{\ol{b},\ol{c}}$ is a non-abelian group of order $qr$. Without loss of generality, we assume that $\gen{\ol{b},\ol{c}}=\gen{\ol{b}}\rtimes\gen{\ol{c}}$, hence $r\mid q-1$. Now, by invoking Lemma \ref{d(K,G)<=d(H,G)}, we can show that
\[d(C_G(a),G)<d(\gen{b,c},G)<d(\gen{a},G),\] 
which yields $d(\gen{x},G)=d(\gen{b,c},G)$. On the other hand,
\begin{align*}
d(\gen{x},G)&=\frac{|G|+(p-1)|C_G(a)|+(q-1)|C_G(b)|+(p-1)(q-1)|C_G(x)|}{pq|G|}\\
&=\frac{|G|+(p-1)pqr|Z|+(q-1)|C_G(b)|+(p-1)(q-1)pq|Z|}{pq|G|}
\end{align*}
and
\begin{align*}
d(\gen{b,c},G)&=\frac{|G|+(q-1)|C_G(b)|+q(r-1)|C_G(c)|}{qr|G|},
\end{align*}
from which in conjunction with the fact that $\gen{a}$, $\gen{b}$, and $\gen{c}$ have the same relative commutativity degrees, we obtain
\[|\ol{G}|=pr\cdot\frac{pqr+p-pr-qr}{p-r}.\]
Since $q$ divides $|\ol{G}|$, it follows that $q\mid r-1$, which contradicts our earlier result that $r\mid q-1$. Therefore $\pi(\ol{C_G(a)})=\{p,q\}$ and $|\ol{C_G(a)}|=p^mq^n$ for some $m,n\geq1$. 

Let $\ol{P}$ and $\ol{Q}$ be a Sylow $p$-subgroup and a Sylow $q$-subgroup of $\ol{C_G(a)}$, respectively. We show that $|\ol{Q}|=q$. Suppose on the contrary that $|\ol{Q}|>q$ and $\ol{Q}_0$ is a subgroup of $\ol{Q}$ of order $q^2$ containing $\ol{b}$. Notice that all non-central elements of $Q_0$ have the same centralizer sizes. If $d(\gen{x},G)=d(Q_0,G)$, then as $d(\gen{a},G)=d(\gen{b},G)$, and 
\[d(\gen{x},G)=\frac{|G|+(p-1)|C_G(a)|+(q-1)|C_G(b)|+(p-1)(q-1)|C_G(x)|}{pq|G|}\]
and 
\[d(Q_0,G)=\frac{|G|+(q^2-1)|C_G(b)|}{q^2|G|},\]
we obtain
\[|\ol{G}|=pq\cdot\frac{pq^2-pq-q^2+p}{p-q}.\]
Hence $|\ol{G}|$ is not divisible by $q^2$ contradicting our assumption. Thus $d(\gen{x},G)\neq d(Q_0,G)$, and by Lemma \ref{HnonabelianKabelian},
\[d(G)<d(C_G(a),G)<d(\gen{x},G),\ d(Q_0,G)<d(\gen{b},G)<1,\]
which is a contradiction. Therefore $|\ol{Q}|=q$ and subsequently $|\ol{P}|>p$ by Theorem \ref{|D(G)|=3}(ii). Let $\ol{H}$ be a subgroup of $\ol{P}$ containing $\ol{a}$ properly. Since $\gen{a}\leq H\leq P\leq C_G(a)$, we have 
\[d(C_G(a),G)\leq d(P,G)\leq d(H,G)\leq d(\gen{a},G).\]
If $d(H,G)=d(\gen{a},G)$, then 
\[H=\gen{a}C_H(x)=\gen{a}\gen{Z(G),a}=\gen{Z(G),a}\]
by Lemma \ref{d(K,G)<=d(H,G)}, which is a contradiction. Also, if $d(P,G)=d(C_G(a),G)$, then, by applying Lemma \ref{d(K,G)<=d(H,G)} once more, it follows that $C_G(a)=PC_{C_G(a)}(g)=P$ for any $\ol{g}\in\ol{P}\setminus\gen{\ol{a}}$, a contradiction. Notice that $C_G(x)=\gen{Z(G),x}$ and hence $C_P(b)=\gen{Z(G),a}$. Thus 
\[d(G)<d(C_G(a),G)<d(P,G)\leq d(H,G)<d(\gen{a},G)<1,\]
from which we obtain $d(P,G)=d(H,G)=d(\gen{x},G)$. As a result, $P=HC_P(b)=H\gen{Z(G),a}=H$, which implies that $|\ol{P}|=p^2$ and hence $\ol{|C_G(a)|}=p^2q$. Furthermore, $\ol{P}$ is non-cyclic otherwise the equalities $d(P,G)=d(\gen{x},G)$ and $d(\gen{a},G)=d(\gen{b},G)$ result in a contradiction. If $\ol{a'}\in \ol{P}\setminus\{\ol{1}\}$ is such that $|C_G(a')|\neq|C_G(a)|$, then we must have $d(\gen{a'},G)=d(\gen{x},G)$ for
\[d(G)<d(C_G(a),G)<d(\gen{x},G)<d(\gen{a},G)<1\]
and $d(C_G(a),G)<d(\gen{a'},G)\neq d(\gen{a},G)$. It follows that $|\ol{G}|\leq p^2(2q-1)$, which contradicts the fact that $|\ol{G}|\geq2|\ol{C_G(a)}|=2p^2q$. Thus $|C_G(a')|=|C_G(a)|$ for all $\ol{a'}\in \ol{P}\setminus\{\ol{1}\}$. Now, the equalities $d(P,G)=d(\gen{x},G)$ and $d(\gen{a},G)=d(\gen{b},G)$ leads us to the final contradiction. The proof is complete.
\end{proof}

Having proved the above major lemma, we need yet to state two rather easy related results before proving our main classification theorem.
\begin{lemma}\label{d(<x>,G)<>d(<y>,G)}
Let $G$ be a finite group and $\ol{x},\ol{y}\in\ol{G}$ be elements of distinct prime orders $p$ and $q$ such that $\ol{C_G(x)}$ and $\ol{C_G(y)}$ have prime power orders. Then either $\ol{G}$ is a group of order $pq$ or $d(\gen{x},G)\neq d(\gen{y},G)$.
\end{lemma}
\begin{proof}
If the equality $d(\gen{x},G)=d(\gen{y},G)$ holds, then 
\[q(|\ol{G}|+(p-1)|\ol{C_G(x)}|)=p(|\ol{G}|+(q-1)|\ol{C_G(y)}|).\]
Assume $p>q$. As $|\ol{C_G(x)}|$ divides the right hand side of the above equality, it follows that $|\ol{C_G(x)}|$ divides $p$, hence $|\ol{C_G(x)}|=p$. Then
\[|\ol{G}|=pq\left(1-\frac{(q-1)(|\ol{C_G(y)}|/q-1)}{p-q}\right)<pq,\]
if $|\ol{C_G(y)}|>q$. Thus $|\ol{C_G(y)}|=q$, which implies that $|\ol{G}|=pq$.
\end{proof}

The following lemma gives a partial answer to Conjecture \ref{|D(G)|>=l_M(G/Z(G))+1?}. Note that the groups $G/Z(G)$ in the lemma are classified in \cite{wb-gt}.
\begin{lemma}\label{|D(G)|>=Omega(G/Z(G))+1}
Let $G$ be a finite non-nilpotent group such that nontrivial elements of $G/Z(G)$ have prime power orders. Then $|\D(G)|\geq l_M(G/Z(G))+1$.
\end{lemma}
\begin{proof}
Let $H$ and $K$ be subgroups of $G$ such that $Z(G)\leq H<K$. If $d(H,G)=d(K,G)$, then $K=HC_K(g)$ for all $g\in G$. Let $p$ be a prime divisor of $\ol{G}$ such that either $p\nmid [K:H]$ or $[K:H]$ is divisible by $pq$ for some prime $q\neq p$. If $\ol{g}\in\ol{G}$ is a $p$-element, then $\ol{C_K(g)}$ is a $p$-subgroup of $\ol{K}$ so that $|HC_K(g)|=|H|[C_K(g):C_H(g)]\neq |K|$. Thus $d(K,G)<d(H,G)$ by Lemma \ref{d(K,G)<=d(H,G)}, from which the result follows.
\end{proof}

Now, we are able to classify all non-nilpotent finite groups with five relative commutativity degrees.
\newline\newline
\textbf{\textit{Proof of Theorem \ref{non-nilpotent}.} }
By Lemma \ref{G/Z(G) has prime power order elements}, we know that non-trivial elements of $\ol{G}$ have non-cubic prime power orders. First we show that $\ol{G}$ is a $\{p,q\}$-group. Suppose on the contrary that $|\ol{G}|$ is divisible by three distinct primes $p,q,r$ and $\ol{a},\ol{b},\ol{c}$ are elements of $\ol{G}$ of orders $p,q,r$, respectively. By Lemma \ref{d(<x>,G)<>d(<y>,G)}, we can assume that 
\[d(G)<d(\gen{a},G)<d(\gen{b},G)<d(\gen{c},G)<1.\]
By Lemmas \ref{d(K,G)<=d(H,G)} and \ref{CHgABH}, $\gen{\ol{a}}$ is a maximal subgroup of $\ol{G}$, which implies that $\ol{G}$ is a Frobenius group whose kernel and complements are both prime power groups, a contradiction. Thus $\ol{G}$ is a $\{p,q\}$-group so that $|\ol{G}|=p^mq^n$ for some $m,n\geq1$. Furthermore, as $|\ol{G}|\neq pq$, Lemma \ref{|D(G)|>=Omega(G/Z(G))+1} yields $m+n=3$ or $4$. Let $P$ and $Q$ be a Sylow $p$-subgroup and a Sylow $q$-subgroup of $G$, respectively.

First assume that $|\ol{G}|=p^2q^2$. Then all non-central $p$-elements (resp. $q$-elements) of $G$ have the same centralizer sizes, say $p^c|Z|$ for some $c\in\{1,2\}$ (resp. $q^d|Z|$ for some $d\in\{1,2\}$). If $\ol{P^*}$ and $\ol{Q^*}$ denote maximal subgroups of $\ol{P}$ and $\ol{Q}$, respectively, then at least two numbers among
\[d(P,G),\ d(P^*,G),\ d(Q,G),\ d(Q^*,G)\]
must be equal. Examining all possible cases, it yields $c=d=2$ and $d(P,G)=d(Q,G)$. Hence $P$ and $Q$ are abelian. Then $P\cap P^g=Q\cap Q^{g'}=Z(G)$ for all $g\in G\setminus N_G(P)$ and $g'\in G\setminus N_G(Q)$. If $q^j=[G:N_G(P)]$ and $p^i=[G:N_G(Q)]$, then as conjugates of $\ol{P}\setminus\{\ol{1}\}$ and $\ol{Q}\setminus\{\ol{1}\}$ partition $\ol{G}\setminus\{\ol{1}\}$, we must have
\[p^2q^2-1=q^j(p^2-1)+p^i(q^2-1),\]
which has no solutions for primes $p$ and $q$ assuming that $i,j\in\{1,2\}$. Therefore, $|\ol{G}|\neq  p^2q^2$ and we can assume that $|\ol{Q}|=q$. In what follows, $\ol{P_i}$ stands for any subgroup of $\ol{P}$ of order $p^i$ for $i=1,\ldots,m$.

From the proof Lemma \ref{|D(G)|>=Omega(G/Z(G))+1}, it follows that all non-central $p$-elements of $G$ have the same centralizer size, say $p^c|Z|$. Clearly, $P$ is abelian if $c\geq2$ so that either $p^c=p$ or $p^c=p^m$. A simple verification shows that for a subgroup $P^*$ of $P$ we have $d(P^*,G)=d(Q,G)$ if and only if $P=P^*$ is abelian. Hence $P$ is abelian when $|\ol{P}|=p^3$ otherwise the following elements
\[d(G),\ d(P_3,G),\ d(P_2,G),\ d(P_1,G),\ d(Q,G),\ 1\]
of $\D(G)$ are pairwise distinct, which is a contradiction.

Suppose $N_G(Q)\neq Q$. If $H$ denotes a subgroup of $N_G(Q)$ such that $|\ol{H}|=pq$, then simple computations show that $d(H,G)\neq d(P_i,G)$ for $i=1,\ldots,m$. Since $d(G)<d(P_m,G)<\cdots<d(P_1,G)<1$, it follows that $m=2$. In particular, we must have $d(P,G)=d(Q,G)$, which implies that $P$ is abelian as mentioned in the previous paragraph. Since $\ol{H}\cong C_q\rtimes C_p$ is non-abelian, we have $p<q$ so that $H\normal G$ as $[\ol{G}:\ol{H}]$ is the smallest prime dividing $|\ol{G}|$. Hence $\ol{Q}\normal\ol{G}$. The fact that $\Aut(\ol{Q})$ is cyclic and $C_{\ol{G}}(\ol{Q})=\ol{Q}$ implies that $\ol{P}$ is cyclic. Therefore $G$ is a group as in part (i).

Now, assume that $N_G(Q)=Q$. Then $\ol{G}=\ol{P}\rtimes\ol{Q}$ is a Frobenius group. Assume $\ol{G}$ has a normal subgroup $\ol{P^*}$ of order $p^k$ for some $1\leq k<m$. Then 
\[d(P^*Q,G)=\frac{p^mq+(p^k-1)p^c+p^k(q-1)q}{p^{k+m}q^2}\neq\frac{p^mq+(p^i-1)p^c}{p^{i+m}q}=d(P_i,G)\]
for $i=1,\ldots,m$. Since $d(G)<d(P_m,G)<\cdots<d(P_1,G)<1$, we must have $m=2$ and $k=1$. Furthermore,
as the elements
\[d(G),\ d(P,G),\ d(P^*Q,G),\ d(P^*,G),\ 1\]
of $\D(G)$ are pairwise distinct and $d(Q,G)\neq d(P^*,G),d(P^*Q,G)$ by Lemmas \ref{d(<x>,G)<>d(<y>,G)} and \ref{|D(G)|>=Omega(G/Z(G))+1}, we must have $d(Q,G)=d(P,G)$, which yields $P$ is abelian as mentioned above. Hence $G$ is a group as in parts (ii) or (iii).

Finally, assume that $\ol{P}$ is a minimal normal subgroup of $\ol{G}$. Then $\ol{P}$ is an elementary abelian $p$-group. If $|\ol{P}|=p^2$, then part (iii) and Theorem \ref{|D(G)|=4}(ii)  yield $c=1$ and $P$ is non-abelian, which implies that $G$ is a group as in part (iv). Now, assume that $|\ol{P}|=p^3$. Then
\[d(P_i,G)=\frac{|Z||G|+(p^i-1)p^c|Z|^2}{p^i|Z||G|}=\frac{p^{3-c}q+p^i-1}{p^{3+i-c}q}\]
for $i=1,2,3$. On the other hand, we must have $d(Q,G)=d(P_i,G)$ for some $i\in\{2,3\}$ as $d(Q,G)\neq d(P_1,G)$ by Lemma \ref{d(<x>,G)<>d(<y>,G)}. A simple verification shows that $i=c=3$, hence $P$ is abelian and $G$ is a group as in part (v). The proof is complete.
\hfill$\Box$


\begin{thebibliography}{0}
\bibitem{wb-gt}
W. Bannuscher and G. Tiedt, On a theorem of Deaconescu, \textit{Rostock. Math. Kolloq.} \textbf{47} (1994), 23--26.

\bibitem{fm-dm-ans}
F. Barry, D. MacHale, and $\acute{A}$. N$\acute{i}$Sh$\acute{e}$, Some supersolvability conditions for finite groups, \textit{Math. Proc. R. Ir. Acad.} \textbf{106}A(2) (2006), 163--177.

\bibitem{rb-ae-mfdg}
R. Barzegar, A. Erfanian, and M. Farrokhi D. G., Finite groups with three relative commutativity degrees, \textit{Bull. Iranian Math. Soc.} \textbf{39}(2) (2013), 271--280.

\bibitem{se}
S. Eberhard, Commuting probabilities of finite groups, \textit{Bull. Lond. Math. Soc.} \textbf{47}(5) (2015), 796--808. 

\bibitem{pe-pt}
P. Erd\"{o}s and P. Turan, On some problems of a statistical group-theory, IV, \textit{Acta Math. Acad. Sci. Hungary} \textbf{19} (1968), 413--435.

\bibitem{ae-mfdg}
A. Erfanian and M. Farrokhi D. G., Finite groups with four relative commutativity degrees, \textit{Algebra Colloq.} \textbf{22}(3) (2015), 449--458.

\bibitem{ae-rr-pl}
A. Erfanian, R. Rezaei, and P. Lescot, On the relative commutativity degree of a subgroup of a finite group, \textit{Comm. Algebra} \textbf{35}(12) (2007), 4183--4197.

\bibitem{ive-bs}
I. V. Erovenko and B. Sury, Commutativity degrees of wreath products of finite abelian groups, \textit{Bull. Aust. Math. Soc.} \textbf{77}(1) (2008), 31--36. 

\bibitem{mfdg-hs}
M. Farrokhi D. G. and H. Safa, Subgroups with large relative commutativity degree, \textit{Quaest. Math.} \textbf{40}(7) (2017), 973--979.

\bibitem{rmg-grr}
R. M. Guralnik and G. R. Robinson, On the commuting probability in finite groups, \textit{J. Algebra} \textbf{300} (2006), 509--528.

\bibitem{whg}
W. H. Gustafson, What is the probability that two group elements commute?, \textit{Amer. Math. Monthly} \textbf{80} (1973), 1031--1034.

\bibitem{eh-dm-ans}
R. Heffernan, D. Machale, and $\acute{A}$. N$\acute{i}$Sh$\acute{e}$, Restrictions on commutativity ratios in finite groups, \textit{Int. J. Group Theory} \textbf{3}(4) (2014), 1--12.

\bibitem{ph}	
P. Hegarty, Limit points in the range of the commuting probability function on finite groups, \textit{J. Group Theory} \textbf{16}(2) (2013), 235--247.  

\bibitem{pl}
P. Lescot, Isoclinism classes and commutativity degrees of finite groups, \textit{J. Algebra} \textbf{177} (1995), 847--869.

\bibitem{pl-hnn-yy}
P. Lescot, H. N. Nguyen, and Y. Yang, On the commuting probability and supersolvability of finite groups, \textit{Monatsh. Math.} \textbf{174}(4) (2014), 567--576.

\bibitem{djsr}
D. J. S. Robinson, \textit{A Course in the Theory of Groups}, Second Edition, Spring-Verlag, New York, 1996.

\bibitem{djr}
D. J. Rusin, What is the probability that two elements of a finite group commute?, \textit{Pacific J. Math.} \textbf{82} (1979), 237--247.
\end{thebibliography}
\end{document}